\documentclass[amstex,12pt,reqno]{amsart}
\usepackage{amsmath,amsfonts,amssymb,amsthm,enumerate,hyperref,multicol}
\usepackage{graphicx}

\newtheorem{lemma}{Lemma}
\newtheorem{thm}{Theorem}

\newtheorem{remark}{Remark}
\newtheorem{tb}{Table}

\numberwithin{equation}{section}
\begin{document}

\leftline{ \scriptsize \it  }
\title[]
{Approximation of Discontinuous Signals by Exponential Sampling Series}
\maketitle

\begin{center}
{\bf SATHISH KUMAR ANGAMUTHU $^1,$ PRASHANT KUMAR$^2$ and DEVARAJ PONNAIAN$^3$}\\

\vskip0.2in
$^{1,2}$Department of Mathematics, Visvesvaraya National Institute of Technology Nagpur, Nagpur, Maharashtra-440010, India\\
\verb"mathsatish9@gmail.com" and \verb''pranwd92@gmail.com

\vskip0.2in
$^3$School of Mathematics, Indian Institute of Science Education and Research, Thiruvananthapuram, India.\\
\verb"devarajp@iisertvm.ac.in"
\end{center}

\begin{abstract}
We analyse the behaviour of the exponential sampling series $S_{w}^{\chi}f$ at jump discontinuity of the bounded signal $f.$ We obtain a representation lemma that is used for analysing the series $S_{w}^{\chi}f$ and we establish approximation of jump discontinuity functions by the series $S_{w}^{\chi}f.$ The rate of approximation of the exponential sampling series $S_{w}^{\chi}f$ is obtained in terms of logarithmic modulus of continuity of functions and the round-off and time-jitter errors are also studied. Finally we give some graphical representation of approximation of discontinuous functions by $S_{w}^{\chi}f$ using suitable kernels.
\\
\vskip0.001in

\end{abstract}

\section{Introduction and Preliminaries}
Let $\mathbb{R}^{+}$ denote the set of all positive real numbers and let $ \chi$ be a real valued function defined on $\mathbb{R}^{+}.$ For $\nu \in \mathbb{N}_{0}=\mathbb{N}\cup\{0\},$ the algebraic moments of order $\nu$ is defined by 
$$ m_{\nu}(\chi,u):= \sum_{k= - \infty}^{+\infty}  \chi(e^{-k} u) (k- \log u)^{\nu}, \hspace{0.5cm} \forall \ u \in \mathbb{R}^{+}.$$
In a similar way, we can define the absolute moment of order $\nu$ as
$$ M_{\nu}(\chi,u):= \sum_{k= - \infty}^{+\infty}  |\chi(e^{-k} u)| |k- \log u|^{\nu},  \hspace{0.5cm} \forall \ u \in \mathbb{R}^{+}.$$
We define $ \displaystyle M_{\nu}(\chi):= \sup_{u \in \mathbb{R}^{+}} M_{\nu}(\chi,u). $ We say that $ \chi$ is a kernel if it satisfies the following conditions: 
\begin{itemize}
\item[(i)] for every $ u \in \mathbb{R}{^{+}},$ $\displaystyle  \sum_{k=- \infty}^{+\infty} \chi(e^{-k}u) =1,$

\item[(ii)]for some $\nu>0,$  $\displaystyle M_{\nu}(\chi,u)=\sup_{u \in \mathbb{R}^{+}}  \sum_{k= - \infty}^{+\infty}  |\chi(e^{-k} u)| |k- \log u|^{\nu}<+\infty.$
\end{itemize}
Let $\Phi$ denote the set of all functions satisfying conditions (i) and (ii). For $ t \in \mathbb{R}^{+}, $ $\chi\in\Phi$ and $ w >0,$ the exponential sampling series for a function $ f:\mathbb{R}^{+} \rightarrow \mathbb{R} $ is defined by (\cite{bardaro7})
\begin{equation} \label{classical}
(S_{w}^{\chi}f)(t)= \sum_{k=- \infty}^{+\infty} \chi(e^{-k} t^{w}) f( e^{\frac{k}{w}}).
\end{equation}
It is easy to see that the series $S_{w}^{\chi}f$ is well defined for $f\in L^{\infty}(\mathbb{R}^{+}).$ Using the above sampling series $S_{w}^{\chi}f$ one can reconstruct the functions which are not Mellin-band limited. Recently, Bardaro et.al. \cite{bardaro3}  pointed out that the study of Mellin-band limited functions are different from that of Fourier-band limited functions. Mamedov was the first person who studied the Mellin theory in \cite{mamedeo} and then Butzer et.al.  further developed the Mellin's theory and studied its approximation properties in \cite{butzer3,butzer4,butzer5,butzer7}. The reconstruction using exponential sampling formula was first studied by Butzer and Jansche in \cite{butzer5}. The pointwise and uniform convergence of the series $S_{w}^{\chi}f$ for continuous functions was analysed in \cite{bardaro7} and the convergence of $S_{w}^{\chi}f$ was studied in Mellin-Lebesgue spaces. Recently Bardaro et.al. studied various approximation results using Mellin transform which can be seen in \cite{bardaro1,bardaro9,bardaro2,bardaro3,bardaro7}. To improve the rate of convergence, a linear combination of $S_{w}^{\chi}f$ was taken in \cite{linear}. The exponential sampling series with the sample points which are exponentially spaced on $\mathbb{R}{^{+}}$ has been obtained as solution of some mathematical model related to light scattering, Fraunhofer diffraction and radio astronomy (see \cite{bert, cas, gori, ost}). 

The approximation of discontinuous functions by classical sampling operators was first initiated by Butzer et.al.\cite{discont}. Further, the Kantorovich sampling series for discontinuous signals was analysed in \cite{Maria}. Inspired by these works we analyse the behaviour of exponential sampling series (\ref{classical}) as $w \rightarrow \infty$ for discontinuous functions at the jump discontinuities, i.e., at a point $t$ where the limits
$$ f(t+0):=\lim_{p \rightarrow 0^{+}}f(t+p),$$ and $$ f(t-0):=\lim_{p \rightarrow 0^{+}}f(t-p)$$
exists and are different. For a kernel $\chi,$ we define the functions
$$\psi_{\chi}^{+}(u):=\sum _{k<\log u}\chi(ue^{-k}),$$
and
$$ \psi_{\chi}^{-}(u):=\sum _{k>\log u}\chi(ue^{-k}).$$

We observe that $\psi_{\chi}^{+}(u)$ and $\psi_{\chi}^{-}(u)$ are recurrent functions with fundamental interval $[1, e]$. Now we shall recall the definition of Mellin transform. 
Let $ L^p(\mathbb{R}^+),$ $p\in[1,\infty)$ be the set of all Lebesgue measurable and $p$-integrable functions defined on $\mathbb{R}^+.$ For $c \in \mathbb{R}$, we define the space
$$X_c=\{f :\mathbb{R}^+\rightarrow\mathbb{C}:f(\cdot)(\cdot)^{c-1}\in L^1(\mathbb{R}^+)\}$$ equipped with the norm
$$\Vert f \Vert_{X_c}=\Vert f(\cdot)(\cdot)^{c-1} \Vert_1=\int_0^{+\infty} |f(y)|y^{c-1}dy.$$
For $f \in X_c,$ it's Mellin transform is defined by 
$$\widehat{[f]_{M}}(s):= \int_0^{+\infty} y^{s-1}f(y)\ dy \ , \,  \ (s = c + it, t \in \mathbb{R}).$$ We say that a function $f\in X_{c} \cap C(\mathbb{R}^+), c\in\mathbb{R}$ is Mellin band-limited in the interval $[-\kappa, \kappa],$ if $\widehat{[f]_{M}}(c+it)=0$ for all $|t|>\kappa ,\ \kappa\in\mathbb{R}^+.$

The paper is organized as follows. In section 2, we prove the representation lemma for the exponential sampling series (\ref{classical}) and using this lemma we analyse the approximation of discontinuous functions by $S_{w}^{\chi}f$ in Theorem \ref{t2}, Theorem \ref{t3} and Theorem \ref{t5}. Further we analyse the degree of approximation for the sampling series (\ref{classical}) in-terms of logarithmic modulus of smoothness in section 3. In section 4, we study the round-off and time jitter errors for these sampling series. 
In example section, we have given a construction of a family of Mellin-band limited kernels such that $\chi(1)=0$ for which $S_{w}^{\chi}f$ converge at any jump discontinuities. Further, the convergence at discontinuity points of the sampling series $S_w^{\chi}{f}$ has been tested numerically and numerical results are provided in Tables 1, 2 and 3. 

\section{Approximation of Discontinuous Signals}
For any given bounded function $f:\mathbb{R}^{+}\rightarrow\mathbb{R},$ we first prove the following representation lemma for the sampling series $S_{w}^{\chi}f.$ Throughout this section we assume that the right and left limits of $f$ at $t\in\mathbb{R}^{+}$ exist and are finite.

\begin{lemma}\label{lemma1}
For a given bounded function $f:\mathbb{R}^{+}\rightarrow\mathbb{R}$ and a fixed $t\in\mathbb{R}^{+},$ let $h_{t}:\mathbb{R}^{+}\rightarrow\mathbb{R}$ be defined by 
\begin{eqnarray*} \label{M3}
h_{t} (x) =
\left\{
\begin{array}{ll}
{f(x)-f(t-0),} & \mbox{if} \ x <t \\

{f(x)-f(t+0),} &\mbox{if} \  x >t\\

{0,}  &\mbox{if} \ x=t.
\end{array}
\right.
\end{eqnarray*}
Then the following holds:
\begin{eqnarray*}
  (S_{w}^{\chi}f)(t)=(S_{w}^{\chi}h_{t})(t)+f(t-0)+\psi_{\chi}^{-}(t^{w})[f(t+0)-f(t-0)]+\chi(1)[f(t)-f(t-0)],
\end{eqnarray*}
if $w\log(t)\in\mathbb{Z}$ and
\begin{eqnarray*}
(S_{w}^{\chi}f)(t)=(S_{w}^{\chi}h_{t})(t)+f(t-0)+\psi_{\chi}^{-}(t^{w})[f(t+0)-f(t-0)],
\end{eqnarray*}
 if $w\log(t)\notin\mathbb{Z}.$
\end{lemma}

\begin{proof}
Let $w\log(t)\in \mathbb{Z},$ and $w>0.$ Then, we can write
\begin{eqnarray*}
  (S_{w}^{\chi}h_{t})(t)&=&\sum _{k<w\log(t)}\chi(e^{-k} t^{w})(f(e^{\frac{k}{w}})-f(t-0))+\sum _{k> w\log(t)}\chi(e^{-k} t^{w})(f(e^{\frac{k}{w}})-f(t+0)) \\
  &=&\sum _{k<w\log(t)}\chi(e^{-k} t^{w})f(e^{\frac{k}{w}})+\sum _{k\geq w\log(t)}\chi(e^{-k} t^{w})f(e^{\frac{k}{w}})-f(t-0)\sum _{k<w\log(t)}\chi(e^{-k} t^{w})
  \\&&
  -f(t+0)\sum _{k> w\log(t)}\chi(e^{-k} t^{w})-\chi(1)f(t) \\
  &=&(S_{w}^{\chi}f)(t)-f(t-0)\sum _{k<w\log(t)}\chi(e^{-k} t^{w})-f(t+0)\sum _{k> w\log(t)}\chi(e^{-k} t^{w})-\chi(1)f(t).
  \end{eqnarray*}
  Adding and subtracting $\displaystyle f(t-0)\sum _{k\geq w\log(t)}\chi(e^{-k} t^{w})$ in the above equation and rearranging all terms, we obtain
  \begin{eqnarray*}
  (S_{w}^{\chi}f)(t)&=&(S_{w}^{\chi}h_{t})(t)+f(t-0)\bigg(\sum _{k<w\log(t)}\chi(e^{-k} t^{w})+\sum _{k\geq w\log(t)}\chi(e^{-k} t^{w})\bigg)
  \\&&+[f(t+0)-f(t-0)]\sum _{k> w\log(t)}\chi(e^{-k} t^{w})+\chi(1)f(t)-f(t-0)\chi(1)\\
  &=&(S_{w}^{\chi}h_{t})(t)+f(t-0)\sum_{k=- \infty}^{+\infty} \chi(e^{-k} t^{w})+[f(t+0)-f(t-0)]\sum _{k> w\log(t)}\chi(e^{-k} t^{w})\\&&
  \chi(1)[f(t)-f(t-0)].
   \end{eqnarray*}
   Hence using the condition that $\displaystyle \sum_{k=- \infty}^{+\infty} \chi(e^{-k} u) =1,$ we can easily obtain
   \begin{eqnarray*}
  (S_{w}^{\chi}f)(t)&=&(S_{w}^{\chi}h_{t})(t)+f(t-0)+\psi_{\chi}^{-}(t^{w})[f(t+0)-f(t-0)]+\chi(1)[f(t)-f(t-0)].
   \end{eqnarray*}
  Now let $w\log(t)\notin \mathbb{Z}$ and $w>0.$ Then repeating the same computations, we easily obtain
 \begin{eqnarray*}
  (S_{w}^{\chi}f)(t)&=&(S_{w}^{\chi}h_{t})(t)+f(t-0)+[f(t+0)-f(t-0)]\sum _{k\geq w\log(t)}\chi(e^{-k} t^{w})\\
&=&(S_{w}^{\chi}h_{t})(t)+f(t-0)+[f(t+0)-f(t-0)]\psi_{\chi}^{-}(t^{w}).
 \end{eqnarray*}
\end{proof}

Before proving the approximation of discontinuous functions by $S_w^{\chi}{f},$ we recall the following theorem proved in (\cite{bardaro7}) for continuous functions on $\mathbb{R}^{+}.$
\begin{thm}\label{t1}
Let $ f:\mathbb{R}^{+} \rightarrow \mathbb{R} $ be a bounded function and $\chi\in\Phi.$ Then $(S_w^{\chi}{f})(t)$ converges to $f(t)$ at any point $t$ of continuity. Moreover, if $f\in C(\mathbb{R}^{+}),$ then we have 
$$ \lim_{ w \rightarrow \infty} \| f- S_w^{\chi}{f}\|_{\infty} = 0.$$
\end{thm}

Now we analyse the behaviour of the exponential sampling series at jump discontinuity at $t\in\mathbb{R}^{+}$ when $w\log(t)\in\mathbb{Z}.$ 

\begin{thm}\label{t2}
Let $f:\mathbb{R}^{+}\rightarrow\mathbb{R}$ be a bounded signal and let $t\in\mathbb{R}^{+}$ be a point of non-removable jump discontinuity of $f.$ For a given $\alpha \in \mathbb{R},$ the following statements are equivalent:
\begin{enumerate}
\item[(i)] $\displaystyle \lim_{\stackrel{w \rightarrow \infty}{w\log(t)\in \mathbb{Z}}}
(S_w^{\chi}f)(t)=\alpha f(t+0)+[1-\alpha-\chi(1)]f(t-0)+\chi(1)f(t), \hspace{0.25cm}$\\

\item[(ii)] $\psi_{\chi}^{-}(1)=\alpha,$  \\

\item[(iii)] $\psi_{\chi}^{+}(1)=1-\alpha-\chi(1). $
\end{enumerate}
\end{thm}

\begin{proof}
First, we prove that $(i)\Longleftrightarrow(ii)$. In view of the representation Lemma \ref{lemma1}, we have
\begin{eqnarray*}
  (S_{w}^{\chi}f)(t)=(S_{w}^{\chi}h_{t})(t)+f(t-0)+\psi_{\chi}^{-}(t^{w})[f(t+0)-f(t-0)]+\chi(1)[f(t)-f(t-0)],
   \end{eqnarray*}
for any $w>0$ such that $w\log(t)\in \mathbb{Z}.$ Since $h_{t}$ is bounded and continuous at zero and using Theorem \ref{t1}, we obtain
$$\displaystyle\lim_{w \rightarrow \infty }(S_{w}^{\chi}h_{t})(t)=0.$$ Thus, we have
\begin{eqnarray*}
\displaystyle \lim_{\stackrel{w \rightarrow \infty}{w\log(t)\in \mathbb{Z}}}
(S_w^{\chi}f)(t)=f(t-0)+\left(\displaystyle \lim_{\stackrel{w \rightarrow \infty}{w\log(t)\in \mathbb{Z}}} \psi_{\chi}^{-}(t^{w})\right) [f(t+0)-f(t-0)]+\chi(1)[f(t)-f(t-0)].
\end{eqnarray*}
Now, we have
$$\displaystyle \lim_{\stackrel{w \rightarrow \infty}{w\log(t)\in \mathbb{Z}}} \psi_{\chi}^{-}(t^{w})
=\displaystyle \lim_{\stackrel{w \rightarrow \infty}{w\log(t)\in \mathbb{Z}}} \left(\sum _{k>w\log(t)}\chi(e^{-k} t^{w})\right).$$
As $\psi_{\chi}^{-}$ is recurrent with fundamental domain $[1, e],$ we get 
$$\psi_{\chi}^{-}(t^{w})=\psi_{\chi}^{-}(1), \,\ \forall w, t \,\ \mbox{such that}  \,\ w\log(t)\in \mathbb{Z}.$$
Therefore, we have
\begin{eqnarray*}
\displaystyle \lim_{\stackrel{w \rightarrow \infty}{w\log(t)\in \mathbb{Z}}}
(S_w^{\chi}f)(t)=\psi_{\chi}^{-}(1) f(t+0)+[1-\psi_{\chi}^{-}(1)-\chi(1)]f(t-0)+\chi(1)f(t).
\end{eqnarray*}

\indent $\mbox{Now}\,\  (i) \Longleftrightarrow
\alpha f(t+0)+[1-\alpha-\chi(1)]f(t-0)+\chi(1)f(t)$
\begin{eqnarray*}
&=&\psi_{\chi}^{-}(1) f(t+0)+[1-\psi_{\chi}^{-}(1)-\chi(1)]f(t-0)+\chi(1)f(t)\\
&\Longleftrightarrow& \psi_{\chi}^{-}(1)(f(t+0)-f(t-0))=\alpha(f(t+0)-f(t-0))\\
&\Longleftrightarrow& \psi_{\chi}^{-}(1)=\alpha \\
&\Longleftrightarrow& (ii)\,\ \mbox{holds}.
 \end{eqnarray*}
   Since $\displaystyle\sum_{k=- \infty}^{+\infty} \chi(e^{-k} t^{w})=1,$ we have
   $$\psi_{\chi}^{+}(1)=1-{\chi}(1)-\psi_{\chi}^{-}(1).$$
   This implies that $(ii)\Longleftrightarrow(iii).$ Hence, the proof is completed.
\end{proof}

Next we analyse the behaviour of the exponential sampling series at jump discontinuity at $t\in\mathbb{R}^{+}$ when $w\log(t)\notin \mathbb{Z}.$

\begin{thm}\label{t3}
Let $f:\mathbb{R}^{+}\rightarrow\mathbb{R}$ be a bounded signal and let $t\in\mathbb{R}^{+}$ be a point of non-removable jump discontinuity of $f.$ Let $\alpha \in \mathbb{R}.$ Then the following statements are equivalent:
\begin{enumerate}
\item[(i)] $\displaystyle \lim_{\stackrel{w \rightarrow \infty}{w\log(t)\notin \mathbb{Z}}} (S_w^{\chi}f)(t)=\alpha f(t+0)+(1-\alpha)f(t-0),$\\

\item[(ii)] $\psi_{\chi}^{-}(u)=\alpha, \hspace{0.5cm} u \in(1, e) $ \\

\item[(iii)]  $\psi_{\chi}^{+}(u)=1-\alpha, \hspace{0.5cm} u\in(1, e).$

\end{enumerate}
\end{thm}

\begin{proof}
Using representation Lemma \ref{lemma1}, we obtain
\begin{eqnarray*}
 \displaystyle \lim_{\stackrel{w \rightarrow \infty}{w\log(t)\notin \mathbb{Z}}} (S_{w}^{\chi}f)(t)
 =f(t-0)+\left(\displaystyle \lim_{\stackrel{w \rightarrow \infty}{w\log(t)\notin \mathbb{Z}}}\psi_{\chi}^{-}(t^{w})\right)[f(t+0)-f(t-0)]
   \end{eqnarray*}
   \begin{eqnarray*}
   (i)&\Longleftrightarrow &
\alpha f(t+0)+(1-\alpha)f(t-0)=f(t-0)+\left(\displaystyle \lim_{\stackrel{w \rightarrow \infty}{w\log(t)\notin \mathbb{Z}}}\psi_{\chi}^{-}(t^{w})\right)[f(t+0)-f(t-0)]\\
&\Longleftrightarrow &
\alpha [f(t+0)-f(t-0)]=\left(\displaystyle \lim_{\stackrel{w \rightarrow \infty}{w\log(t)\notin \mathbb{Z}}}\psi_{\chi}^{-}(t^{w})\right)[f(t+0)-f(t-0)]\\
&\Longleftrightarrow &
\alpha=\displaystyle \lim_{\stackrel{w \rightarrow \infty}{w\log(t)\notin \mathbb{Z}}}
\psi_{\chi}^{-}(t^{w})\\
&\Longleftrightarrow &
\alpha=\psi_{\chi}^{-}(u),  \hspace{0.5cm} \forall u \in(1, e)\\
&\Longleftrightarrow& (ii)\,\ \mbox{holds}.
\end{eqnarray*}
Let $w\log(t)\notin \mathbb{Z}.$ Then, we have
$$\psi_{\chi}^{+}(t^{w})+\psi_{\chi}^{-}(t^{w})=1.$$ Thus, we obtain
$$(ii)\Longleftrightarrow \psi_{\chi}^{+}(u)=1-\alpha, \hspace{0.5cm} u\in(1, e) .$$
\end{proof}

The results in the above Theorem \ref{t3} was proved by assuming that $\psi_{\chi}^{-}(u)$ is constant on $(1, e).$ In what follows, we show that if $\psi_{\chi}^{-}(u)$ is not a constant on $(1,e),$ then the exponential sampling series can not converge at jump discontinuities. 

\begin{thm}\label{t4}
 Let $\chi$ be a kernel such that $\psi_{\chi}^{-}(u)$ is not constant on $(1,e).$  Let $f:\mathbb{R}^{+}\rightarrow\mathbb{R}$ be a bounded signal with a non-removable jump discontinuity at $t\in\mathbb{R}^{+}.$ Then $(S_w^{\chi}f)(t)$ does not converge  point-wise at $t$.
\end{thm}

\begin{proof}
Suppose not. Then $\displaystyle \lim_{\stackrel{w \rightarrow \infty}{w\log(t)\notin \mathbb{Z}}} (S_w^{\chi}f)(t)=\ell,$ for some $\ell\in\mathbb{R}^{+}.$ By the uniqueness of the limit and Lemma \ref{lemma1}, we obtain
\begin{eqnarray*}
\ell=f(t-0)+\displaystyle\lim_{w \rightarrow \infty }\psi_{\chi}^{-}(t^{w})[f(t+0)-f(t-0)].
\end{eqnarray*}
Since $f(t+0)-f(t-0)\neq 0,$ we obtain
\begin{eqnarray*}
\frac{\ell-f(t-0)}{f(t+0)-f(t-0)}=\displaystyle\lim_{w \rightarrow \infty }\psi_{\chi}^{-}(t^{w}).
\end{eqnarray*}
The above expression gives a contradiction. Indeed, if
\begin{eqnarray*}
\displaystyle\lim_{w \rightarrow \infty }\psi_{\chi}^{-}(t^{w})=C,
\end{eqnarray*}
where $C$ is a constant, then it fails to satisfy that $\psi_{\chi}^{-}$ is recurrent and not a constant, hence the theorem proved.
\end{proof}

Finally, the more general theorem of the exponential sampling series at jump discontinuity at $t\in\mathbb{R}^{+}$ for any bounded signal can be proved. 

\begin{thm}\label{t5}
Let $f:\mathbb{R}^{+}\rightarrow\mathbb{R}$ be a bounded signal and let $t\in\mathbb{R}^{+}$ be a point of non-removable jump discontinuity of $f.$ Let $\alpha \in \mathbb{R}.$ 
Suppose that the kernel $\chi$ satisfies the additional condition that $\chi(1)=0.$  Then, the following statements are equivalent:
\begin{enumerate}
\item[(i)] $\displaystyle \lim_{w \rightarrow \infty}(S_w^{\chi}f)(t)=\alpha f(t+0)+(1-\alpha)f(t-0),$\\

\item[(ii)] $\psi_{\chi}^{-}(u)=\alpha, \hspace{0.5cm} u \in[1, e) $ \\

\item[(iii)]  $\psi_{\chi}^{+}(u)=1-\alpha, \hspace{0.5cm} u\in[1, e).$
\end{enumerate}
Moreover, if in addition we assume that $\chi$ is continuous on $\mathbb{R}^{+},$ then the above statements are equivalent to the following statements: 
\begin{enumerate}
\item[(iv)] 
\begin{eqnarray*}
\int_{0}^{1} \chi (u) u^{2k\pi i}\frac{du}{u}
&=&\left\{
\begin{array}{ll}
0 , &  \ \ \ \mbox{if} \ \ \ k \neq 0 \\
\alpha, & \ \ \ \mbox{if} \ \ \ k= 0
\end{array}
\right.
\end{eqnarray*}
\item[(v)]
\begin{eqnarray*}
\int_{1}^{\infty} \chi (u) u^{2k\pi i }\frac{du}{u}
&=&\left\{
\begin{array}{ll}
0 , &  \ \ \ \mbox{if} \ \ \ k \neq 0 \\
1-\alpha, & \ \ \ \mbox{if} \ \ \ k= 0.
\end{array}
\right.
\end{eqnarray*}
\end{enumerate}
\end{thm}

\begin{proof}
Proceeding along the lines proof of Theorem \ref{t2} and Theorem \ref{t3}, we see that $(i), (ii)$ and $(iii)$ are equivalent. Let $\chi$ be continuous on $\mathbb{R}^{+}$ and let 
\begin{eqnarray*}
\chi_0(u)
&=&\left\{
\begin{array}{ll}
\chi (u) , &  \ \ \ \mbox{for} \ \ \ u<1 \\
0, & \ \ \ \mbox{for} \ \ \ u\geq 1.
\end{array}
\right.
\end{eqnarray*}
Then, we have 
\begin{eqnarray*}
\psi_{\chi}^{-}(u)=\sum _{k>\log u}\chi(ue^{-k})=\sum _{k\in\mathbb{Z}}\chi_0(ue^{-k}).
\end{eqnarray*}
Therefore, $\psi_{\chi}^{-}$ is recurrent continuous function with the fundamental interval $[1,e].$ Using Mellin-Poisson summation formula, we obtain  
\begin{eqnarray*}
 \psi_{\chi}^{-}(u)=\sum_{k= - \infty}^{+\infty}\widehat{[\chi_0]_{M}}(2k \pi i) \ u^{-2 k \pi i}
 =\sum_{k= - \infty}^{+\infty}\left(\int_{0}^{1} \chi (u) u^{2k\pi i}\frac{du}{u}\right)u^{-2 k \pi i}.
\end{eqnarray*}
Therefore, we obtain 
\begin{eqnarray*}
 \psi_{\chi}^{-}(u)&=&\alpha, \forall u\in[1,e)\\
&\Longleftrightarrow & \widehat{[\chi]_{M}}(2k \pi i)
=\left\{
\begin{array}{ll}
0 , &  \ \ \ \mbox{if} \ \ \ k \neq 0 \\
\alpha, & \ \ \ \mbox{if} \ \ \ k= 0
\end{array}
\right.\\
&\Longleftrightarrow &
\int_{0}^{1} \chi (u) u^{2k\pi i}\frac{du}{u}
=\left\{
\begin{array}{ll}
0 , &  \ \ \ \mbox{if} \ \ \ k \neq 0 \\
\alpha, & \ \ \ \mbox{if} \ \ \ k= 0.
\end{array}
\right.
\end{eqnarray*}
This implies that $(ii) \Longleftrightarrow  (iv).$ Finally using the condition 
\begin{eqnarray*}
\sum_{k=- \infty}^{+\infty} \chi(e^{-k}u) =1
\Longleftrightarrow 
\widehat{[\chi]_{M}}(2k \pi i)
=\left\{
\begin{array}{ll}
0 , &  \ \ \ \mbox{if} \ \ \ k \neq 0 \\
1, & \ \ \ \mbox{if} \ \ \ k= 0
\end{array}
\right.\\
\end{eqnarray*}
the equivalence between $(iv)$ and $(v)$ can be established easily. Thus the proof is completed. 
\end{proof}

\begin{remark}
Let $f:\mathbb{R}^{+}\rightarrow\mathbb{R}$ be a bounded signal with a removable discontinuity $t\in\mathbb{R}^{+}$, i.e. $f(t+0)=f(t-0)=\ell.$ Then we have
\begin{enumerate}
\item[(i)]  $\displaystyle \lim_{\stackrel{w \rightarrow \infty}{w\log(t)\in \mathbb{Z}}} (S_w^{\chi}f)(t)=\ell+\chi(1)[f(t)-\ell],$\\

\item[(ii)] $\displaystyle \lim_{\stackrel{w \rightarrow \infty}{w\log(t)\notin \mathbb{Z}}} (S_w^{\chi}f)(t)=\ell,$\\

\item[(iii)] If $\chi(1)=0,$ then $\displaystyle \lim_{w \rightarrow \infty}(S_w^{\chi}f)(t)=\ell.$
\end{enumerate}
\end{remark}

\section{Degree of Approximation}
In this section, we estimate the order of convergence of the exponential sampling series by using the logarithmic modulus of continuity. Let $C(\mathbb{R}^+$) denote the space of all real valued bounded continuous functions on $\mathbb{R}^+$ equipped with the supremum norm $\|f\|_{\infty} := \sup_{x \in \mathbb{R}^+} |f(x)|.$ We say that a function $f: \mathbb{R}^+ \rightarrow \mathbb{R}$ is log-uniformly continuous if the following hold: 
for a given $\epsilon > 0,$ there exists $\delta > 0$ such that $|f(p) -f(q)| < \epsilon$ whenever $| \log p - \log q| < \delta,$ for any $p, q \in \mathbb{R}^{+}.$ The subspace consisting of all bounded log-uniformly continuous functions on $\mathbb{R}^{+}$ is denoted by $\mathcal{C} (\mathbb{R}^+).$ Let $f\in C(\mathbb{R}^{+}).$ Then the logarithmic modulus of continuity is defined by
$$ \omega(f,\delta):= \sup \{|f(p)-f(q)|:\  \mbox{whenever} \  |\log (p)-\log (q)| \leq \delta,\  \ \delta \in \mathbb{R}^{+}\} .$$
The logarithmic modulus of continuity satisfies the following properties:
\begin{itemize}
\item[(a)] $\omega(f, \delta) \rightarrow 0,$ as $\delta\rightarrow 0.$
\item[(b)] $\omega(f, c \delta) \leq (c+1) \omega(f,\delta),$ for every $\delta, c>0.$
\item[(c)] $|f(p)-f(q)|\leq  \omega(f,\delta)\left(1+\dfrac{|\log p-\log q|}{\delta}\right).$
\end{itemize}
Further properties of logarithmic modulus of continuity can be seen in \cite{bardaro9}. In the following theorem, we obtain the order of convergence for the exponential sampling series when $M_{\nu}(\chi)<\infty$ for $0<\nu<1.$

\begin{thm} \label{t8}
Let $\chi\in\Phi$ be a kernel such that $M_{\nu}(\chi)<\infty$ for $0<\nu<1$ and $ f \in \mathcal{C}(\mathbb{R}^+)$. Then for sufficiently large $w>0,$ the following hold: 
 $$ |(S_{w}^{\chi}f)(t) - f(t)| \leq  \omega (f,w^{-\nu})[M_{\nu}(\chi)+2M_{0}(\chi)]+2^{\nu+1}\|f\|_{\infty}M_{\nu}(\chi)w^{-\nu},$$
for every $t\in\mathbb{R}^{+}.$
\end{thm}

\begin{proof}
Let $t\in\mathbb{R}^{+}$ be fixed. Then using the condition $\displaystyle\sum_{k= - \infty}^{+\infty}  \chi(e^{-k} t^w) =1,$ we obtain
\begin{eqnarray*}
 |(S_{w}^{\chi}f)(t) - f(t)| &=& \bigg| \sum_{k= - \infty}^{+\infty} \chi(e^{-k} t^{w})(f(e^{\frac{k}{w}}) - f(t)) \bigg| \\
&\leq& \left( \sum_{\big|k-w \log t\big|< \frac{w }{2}} +\sum_{\big|k- w\log t\big| \geq \frac{w}{2}} \right) \big| \chi(e^{-k} t^{w})\big|
|f(e^{\frac{k}{w}}) - f(t)| \\
&:=& I_{1}+I_{2}.
\end{eqnarray*}
Let $0<\nu<1.$ Then we have
$$\omega\left(f,\Big |\frac{k}{w} - \log t \Big|\right)\leq\omega\left(f,\Big |\frac{k}{w} - \log t \Big|^{\nu}\right). $$ Therefore, using the above inequality and the property $(c),$ we obtain
\begin{eqnarray*}
I_{1}&\leq&\sum_{\big|k-w \log t\big|< \frac{w}{2}}\big| \chi(e^{-k} t^{w})\big| \omega\left(f,\Big |\frac{k}{w} - \log t \Big|^{\nu}\right)\\
&\leq&\sum_{\big|k-w \log t\big|< \frac{w}{2}}\big| \chi(e^{-k} t^{w})\big| \left(1+w^{\nu}\Big |\frac{k}{w} - \log t \Big|^{\nu}\right)\omega(f,w^{-\nu})\\
&\leq&\sum_{\big|k-w \log t\big|< \frac{w}{2}}\big| \chi(e^{-k} t^{w})\big| \omega(f,w^{-\nu})
\\&&+\sum_{\big|k-w \log t\big|< \frac{w}{2}}\big| \chi(e^{-k} t^{w})\big| w^{\nu}\Big |\frac{k}{w} - \log t \Big|^{\nu}\omega(f,w^{-\nu})\\
&\leq&\omega(f,w^{-\nu})\sum_{\big|k-w \log t\big|< \frac{w}{2}}\big| \chi(e^{-k} t^{w})\big| +\omega(f,w^{-\nu})\sum_{\big|k-w \log t\big|< \frac{w}{2}}\big| \chi(e^{-k} t^{w})\big| \Big |k - w\log t \Big|^{\nu}.
\end{eqnarray*}
In view of the conditions $M_{0}(\chi)$ and $M_{\nu}(\chi),$ we easily obtain
$$I_{1}\leq \omega(f,w^{-\nu})[M_{0}(\chi)+M_{\nu}(\chi)].$$
Now we estimate $I_2.$ Since $\big|k- w\log t\big| \geq \frac{w}{2},$ we have
$$\frac{1}{\big|k- w\log t\big|^{\nu}}\leq 2^{\nu}w^{-\nu}, \,\,\ 0<\nu<1.$$
Hence, we obtain
\begin{eqnarray*}
I_{2} &\leq&2\|f\|_{\infty} \sum_{\big|k- w\log t\big| \geq \frac{w}{2}}  \big| \chi(e^{-k} t^{w})\big|\leq
2\|f\|_{\infty} \sum_{\big|k- w\log t\big| \geq \frac{w}{2}} \frac{\big|k- w\log t\big|^{\nu} }{\big|k- w\log t\big| ^{\nu}} \big| \chi(e^{-k} t^{w})\big|\\
&\leq&2^{\nu+1}\|f\|_{\infty}w^{-\nu}\sum_{\big|k- w\log t\big| \geq \frac{w}{2}} \big| \chi(e^{-k} t^{w})\big|\big|k- w\log t\big|^{\nu} \leq
2^{\nu+1}\|f\|_{\infty}w^{-\nu}M_{\nu}(\chi)<\infty.
\end{eqnarray*}
On combining the estimates $I_{1}$ and $I_{2},$ we get the desired estimate.
\end{proof}

\section{Round-Off and Time Jitter Errors}
This section is devoted to analyse round off and time jitter errors connected with exponential sampling series (\ref{classical}). The
round-off error arises when the exact sample values $ f( e^{\frac{k}{w}})$ are replaced by approximate close
ones $\bar{f}( e^{\frac{k}{w}})$ in the sampling series (\ref{classical}). Let $\xi_{k}=f( e^{\frac{k}{w}})-\bar{f}( e^{\frac{k}{w}})$ be uniformly bounded by $\xi,$ i.e., $\mid\xi_{k}\mid\leq\xi,$ for some $\xi>0.$ We are interested in analysing the error when $f(t)$ is approximated by the following exponential sampling series:
\begin{eqnarray*}
(S_w^{\chi}\bar{f})(t)=\sum_{k=- \infty}^{+\infty} \chi(e^{-k} t^{w}) \bar{f}( e^{\frac{k}{w}}).
\end{eqnarray*}
The total round-off or quantization error is defined by 
$$(Q_{\xi}f)(t):=\mid (S_w^{\chi}{f})(t)- (S_w^{\chi}\bar{f})(t)\mid.$$

\begin{thm}\label{t9}
For $f\in C(\mathbb{R}^{+}),$ the following hold: 
\begin{enumerate}
\item[(i)] $\parallel (Q_{\xi}f)\parallel_{ C(\mathbb{R}^{+})}\leq\xi M_{0}(\chi)$\\

\item[(ii)] $\parallel f- S_w^{\chi}\bar{f}\parallel_{ C(\mathbb{R}^{+})} \leq C\omega\left(f,\dfrac{1}{w}\right)+\xi M_{0}(\chi),$ 
where $C=M_{0}(\chi) +{ M_{1}(\chi)}.$
\end{enumerate}
\end{thm}

\begin{proof}
The error in the approximation can be splitted as
\begin{eqnarray*}
\mid f(t)- (S_w^{\chi}\bar{f})(t)\mid&\leq & \mid f(t)- (S_w^{\chi}f)(t)\mid+(Q_{\xi}f)(t):=I_1+(Q_{\xi}f)(t). 
\end{eqnarray*}
The term $I_1$ is the error arising if the actual sample value is used and the total round-off or quantization error can be evaluated by 
\begin{eqnarray*}
\parallel (Q_{\xi}f)\parallel_{ C(\mathbb{R}^{+})}&=&\displaystyle\sup_{t\in \mathbb{R}^{+}}{\left|\sum_{k=- \infty}^{+\infty} \chi(e^{-k} t^{w}) {f}( e^{\frac{k}{w}})-\sum_{k=- \infty}^{+\infty} \chi(e^{-k} t^{w}) \bar{f}( e^{\frac{k}{w}})\right|}\\
&=&\displaystyle\sup_{t\in \mathbb{R}^{+}}{\sum_{k=- \infty}^{+\infty}\bigg|\xi_{k}\chi(e^{-k} t^{w})\bigg|}\leq\xi M_{0}(\chi).
\end{eqnarray*}
In view of Theorem 4 (\cite{bardaro7}, page no.7), we have
$$I_1\leq M_{0}(\chi) \omega(f,\delta)+\frac{ \omega(f,\delta)}{w\delta}M_{1}(\chi).$$
On combining the estimates $I_1$ and $I_2,$ we get
\begin{eqnarray*}
\parallel f- S_w^{\chi}\bar{f}\parallel_{ C(\mathbb{R}^{+})} &\leq &\bigg(M_{0}(\chi) +\frac{ M_{1}(\chi)}{w\delta}\bigg)\omega(f,\delta)+\xi M_{0}(\chi).
\end{eqnarray*}
Choosing $\delta=\displaystyle\frac{1}{w},$ we obtain
\begin{eqnarray*}
\parallel f- S_w^{\chi}\bar{f}\parallel_{ C(\mathbb{R}^{+})} &\leq &C\omega(f,\frac{1}{w})+\xi M_{0}(\chi),
\end{eqnarray*}
where $C=M_{0}(\chi) +{ M_{1}(\chi)}.$ Hence, the proof is completed.
\end{proof}

The time-jitter error occurs when the function $f(t)$ being approximated from samples which are taken at perturbed nodes, i.e., the exact sample values $ f( e^{\frac{k}{w}})$ are replaced by $ f( e^{\frac{k}{w}}+\varrho_{k})$ in the sampling series (\ref{classical}). So we are interested in analysing time jitter error and the approximation behaviour when $f(t)$ is approximated by the sampling series 
$\displaystyle\sum_{k=- \infty}^{+\infty} \chi(e^{-k} t^{w}){f}( e^{\frac{k}{w}}+\varrho_{k}).$ We assume that the values $\varrho_{k}$ are bounded by a small number $\varrho,$ i.e., $\mid\varrho_{k}\mid\leq \varrho,$ for all $k\in\mathbb{Z}$ and for some $\varrho>0.$ The total time jitter error is defined by
  $$J_{\varrho}f(t):=\left|\sum_{k=- \infty}^{+\infty} \chi(e^{-k} t^{w}){f}( e^{\frac{k}{w}})-\sum_{k=- \infty}^{+\infty} \chi(e^{-k} t^{w}){f}( e^{\frac{k}{w}}+\varrho_{k})\right|.$$

\begin{thm}\label{t10}
For $f\in C^{(1)}(\mathbb{R}^{+}),$ the following hold: 
\begin{enumerate}
\item[(i)] $\|J_{\varrho} f\|_{ C(\mathbb{R}^{+})}\leq\varrho\, \parallel f^{'}\parallel_{ C(\mathbb{R}^{+})}M_{0}(\chi)
$\\

\item[(ii)] $\bigg\| f(.)-\sum_{k=- \infty}^{+\infty} \chi(e^{-k} (.)^{w}){f}( e^{\frac{k}{w}}+\varrho_{k})\bigg\|_{ C(\mathbb{R}^{+})}\leq \,\
C\omega\left(f,\dfrac{1}{w}\right)+\varrho\parallel f^{'}\parallel_{ C(\mathbb{R}^{+})}M_{0}(\chi),$ where $C=M_{0}(\chi) +{ M_{1}(\chi)}.$
\end{enumerate}
\end{thm}

\begin{proof}
Applying the mean value theorem, error can be estimated by
\begin{eqnarray*}
   \|J_{\varrho} f\|_{ C(\mathbb{R}^{+})}&\leq&\displaystyle\sup_{k\in \mathbb{Z}}\{\sup_{t\in \mathbb{R}^{+}}|{f}( e^{\frac{k}{w}})-{f}( e^{\frac{k}{w}}+\varrho_{k})|\}\sup_{t\in \mathbb{R}^{+}}{\sum_{k=- \infty}^{+\infty}| \chi(e^{-k} t^{w})|}\\
   &\leq& \mid\varrho_{k}\parallel f^{'}\parallel_{ C(\mathbb{R}^{+})}\mid\sup_{t\in \mathbb{R}^{+}}{\sum_{k=- \infty}^{+\infty}| \chi(e^{-k} t^{w})|}
 \leq \varrho\, \parallel f^{'}\parallel_{ C(\mathbb{R}^{+})}M_{0}(\chi).
\end{eqnarray*}
From the above estimates it is clear that the jitter error essentially depends on the smoothness of function $f.$ For $f\in C^{(1)}(\mathbb{R}^{+}),$ the associated approximation error is estimated by
\begin{eqnarray*}
\bigg| f(t)-\sum_{k=- \infty}^{+\infty} \chi(e^{-k}t^{w}){f}( e^{\frac{k}{w}}+\varrho_{k})\bigg|&\leq&
\bigg| f(t)-\sum_{k=- \infty}^{+\infty} \chi(e^{-k} t^{w}){f}( e^{\frac{k}{w}})\bigg|\\&&+\bigg| \sum_{k=- \infty}^{+\infty} \chi(e^{-k} t^{w}){f}( e^{\frac{k}{w}})-\sum_{k=- \infty}^{+\infty} \chi(e^{-k} t^{w}){f}( e^{\frac{k}{w}}+\varrho_{k})\bigg|\\
&\leq& |f(t)-S_w^{\chi}{f}(t)|+J_{\varrho}f(t).
\end{eqnarray*}
Again using Theorem 4 (\cite{bardaro7}, page no.7), we have
$$|f(t)-S_w^{\chi}{f}(t)|\leq M_{0}(\chi) \omega(f,\delta)+\frac{ \omega(f,\delta)}{w\delta}M_{1}(\chi).$$
Using the above estimate and $J_{\varrho} f,$ we obtain
\begin{eqnarray*}
\bigg\| f(.)-\sum_{k=- \infty}^{+\infty} \chi(e^{-k} (.)^{w}){f}( e^{\frac{k}{w}}+\varrho_{k})\bigg\|_{ C(\mathbb{R}^{+})}\leq M_{0}(\chi) \omega(f,\delta)+\frac{ \omega(f,\delta)}{w\delta}M_{1}(\chi)+\varrho\parallel f^{'}\parallel_{ C(\mathbb{R}^{+})}M_{0}(\chi).
\end{eqnarray*}
Hence, the proof is completed.
\end{proof}

\section{Examples of the Kernels} 

In this section, we provide certain examples of the kernel functions which will satisfy our assumptions. 
First we give the family of Mellin-B spline kernels. The Mellin B-spline of order $n$ is given by
$$\bar{B}_{n}(x):= \frac{1}{(n-1)!} \sum_{j=0}^{n} (-1)^{j} {n \choose j} \bigg( \frac{n}{2}+\log x-j \bigg)_{+}^{n+1}, \,\  x \in \mathbb{R}^{+}$$
It can be easily seen that $\bar{B}_{n}(x) $ is compactly supported for every $ n \in \mathbb{N}.$ The Mellin transform of $\bar{B}_{n}$ (see \cite{bardaro7}) is
\begin{eqnarray*}
\widehat{[\bar{B}_{n}]_{M}}(c+it) = \bigg( \frac{\sin(\frac{t}{2})}{(\frac{t}{2})} \Bigg)^{n},  \ \ \hspace{0.5cm} t \neq 0.
\end{eqnarray*}
The Mellin's-Poisson summation formula \cite{butzer4}) is defined by
$$ (i)^{j} \sum_{k= - \infty}^{+\infty} \chi(e^{k} x) ( k-\log u)^{j} = \sum_{k= - \infty}^{+\infty}\frac{d^{j}}{dt^{j}}\widehat{[\chi]_{M}}(2k \pi i) \ x^{-2 k \pi i}, \ \ \ \ \ \ \mbox{for} \ k \in \mathbb{Z}.$$
We need the following lemma (see \cite{bardaro7}).
\begin{lemma}\label{l1}
The condition $\displaystyle \sum_{k= - \infty}^{+\infty}  \chi(e^{-k} x^{w})=1$ is equivalent to
\begin{eqnarray*}
\widehat{[\chi]_{M}}(2k \pi i)= \left\{
\begin{array}{ll}
1 , &\mbox{if} \ k=0 \\
0, & \mbox{otherwise.}
\end{array}
\right.
\end{eqnarray*}
Moreover $ m_{j}(\chi,u)=0$ for $ j= 1,2,\cdots, n$ is equivalent to $\dfrac{d^{j}}{dt^{j}}\widehat{[\chi]_{M}}(2k \pi i)= 0$ for $ j= 1,2,\cdots,n$ and $ \forall \ k \in \mathbb{Z}.$
\end{lemma}
Using the above Lemma, we obtain
\begin{eqnarray*}
\widehat{[\bar{B}_{n}]_{M}}(2k \pi i)= \left\{
\begin{array}{ll}
1 , &\mbox{if} \ k=0 \\
0, & \mbox{otherwise.}
\end{array}
\right.
\end{eqnarray*}
Using Mellin's-Poisson summation formula, it is easy to see that $\bar{B}_{n}(x)$ satisfies the condition (i). As $ \bar{B}_{n}(x)$ is compactly supported, the condition (ii) is also satisfied. Next we consider the  Mellin Jackson kernels. For $x\in\mathbb{R}^{+}, \beta\in\mathbb{N}, \gamma\geq 1,$ the Mellin Jackson kernels are defined by 
$$J_{\gamma, \beta}^{-1}(x):=d_{\gamma, \beta}x^{-c} sinc^{2\beta}\left(\frac{\log x}{2\gamma\beta\pi}\right),$$
where 
$$d_{\gamma, \beta}^{-1}:=\int_{0}^{\infty}sinc^{2\beta}\left(\frac{\log x}{2\gamma\beta\pi}\right)\frac{du}{u}.$$ One can easily verify that the Mellin Jackson kernels also satisfies conditions (i) and (ii) (see \cite{bardaro7}).  We can analyse the convergence of the exponential sampling series with jump discontinuity associated with these kernels only for the case given in Theorem \ref{t2} and we observe that $\chi(1)\neq 0$. So Theorem \ref{t5} can not be applied for these kernels. In order to obtain the convergence of the exponential sampling series at jump discontinuity $t\in\mathbb{R}^{+}$ of the given bounded signal $f:\mathbb{R}^{+}\rightarrow \mathbb{R}$, we need to construct suitable kernels. One such construction is given in the following theorem.

\begin{thm}\label{t11}
Let $\chi_{a},\chi_{b}$ be two continuous kernels supported respectively in the intervals ${[e^{-a},e^{a}]}$ and $ {[e^{-b},e^{b}]}.$ Let $\alpha\in\mathbb{R}$ be fixed. 
We define $\chi:\mathbb{R}^{+}\rightarrow \mathbb{R}^{+}$ by
$$\chi{(u)}:=(1-\alpha)\chi_{a}(2ue^{-a-1})+\alpha\chi_{b}(2ue^{b}), \,\,\ u\in\mathbb{R}^{+}.$$
Then $\chi$ is a kernel satisfying conditions (i), (ii) and $\chi(1)=0.$ Moreover, the corresponding exponential sampling series $S_w^{\chi}{f}, w>0$ based upon $\chi$ satisfy (i) of Theorem \ref{t5} with parameter $\alpha$ for a given bounded signal $f:\mathbb{R}^{+}\rightarrow \mathbb{R}$ at any non-removable discontinuity $t\in\mathbb{R}^{+}$ of $f.$ 
\end{thm}

\begin{proof}
The Mellin transform of $\chi(u)$ is 
\begin{eqnarray*}
\widehat{[\chi]_{M}}(s)
&=&\int_{0}^{\infty}(1-\alpha)\chi_{a}(2te^{-a-1})t^{s-1}dt+\int_{0}^{\infty}\alpha\chi_{b}(2te^{b})t^{s-1}dt\\
&=&(1-\alpha)\widehat{[\chi_{a}]_{M}}(s)\left(\frac{e^{(1+a)}}{2}\right)^{s}+\alpha\widehat{[\chi_{b}]_{M}}(s) \left(\frac{e^{-b}}{2}\right)^{s}.
\end{eqnarray*}
 It is simple to check that $\chi$ satisfies condition (ii). Now we show that kernel satisfies the condition (i). We obtain 
\begin{eqnarray*}
\widehat{[\chi]_{M}}(2k\pi i )&=&(1-\alpha)\widehat{[\chi_{a}]_{M}} (2k\pi i)\left(\frac{e^{(1+a)}}{2}\right)^{2k\pi i}+\alpha\widehat{[\chi_{b}]_{M}}(2k\pi i)\left(\frac{e^{-b}}{2}\right)^{2k\pi i}.
\end{eqnarray*}
As $\chi_{a}$ and $\chi_{b}$ satisfies condition (i), we have 
\begin{eqnarray*}
\widehat{[\chi_{a}]_{M}} (2k\pi i )= \widehat{[\chi_{b}]_{M}}(2k\pi i )=\left\{
\begin{array}{ll}
0 , &  \ \ \ \mbox{if} \ \ \ k \neq 0 \\
1, & \ \ \ \mbox{if} \ \ \ k= 0
\end{array}
\right.
\end{eqnarray*}
For suitable choices of $a$ and $b$, we obtain 
\begin{eqnarray*}
\widehat{[\chi]_{M}} (2k\pi i )= \left\{
\begin{array}{ll}
0 , &  \ \ \ \mbox{ if} \ \ \ k \neq 0 \\
1, & \ \ \ \mbox{if} \ \ \ k= 0.
\end{array}
\right.
\end{eqnarray*}
Therefore, $\chi$ satisfies condition (i) and we can easily see that $ \chi(1)=0.$ Now, we obtain 
\begin{eqnarray*}
\int_{0}^{1} \chi (u) u^{2k\pi i -1}du&=&\alpha\int_{0}^{1} \chi_{b}(2e^{b}u)u^{2k\pi i-1}du\\
&=& \alpha\widehat{[\chi_{b}]_{M}}(2k\pi i)\left(\frac{e^{-b2k\pi i}}{2^{2k\pi i}}\right)\\
&=&\left\{
\begin{array}{ll}
0 , &  \ \ \ \mbox{if} \ \ \ k \neq 0 \\
\alpha, & \ \ \ \mbox{if} \ \ \ k= 0.
\end{array}
\right.
\end{eqnarray*}
Therefore, the condition (iv) of Theorem \ref{t5} is satisfied, hence the proof is completed. 
\end{proof}

Now we test numerically the approximation of discontinuous function
\begin{eqnarray*}
f(t)&=&\left\{
\begin{array}{ll}
\dfrac{11}{2t^2+1} , &  \ \ \ \ \ t<\dfrac{3}{2} \\
\vspace{0.25 cm}
3, & \ \ \  \ \ \ \dfrac{3}{2}\leq t<\dfrac{7}{2} \\
\vspace{0.25 cm}

2, & \ \ \  \ \ \ \dfrac{7}{2}\leq t<\dfrac{11}{2} \\
\vspace{0.25 cm}

\dfrac{12}{1+2t}, & \ \ \  \ \ \ t\geq \dfrac{11}{2} 
\end{array}
\right.
\end{eqnarray*}
by exponential sampling series at its jump discontinuities $t=\dfrac{3}{2},$ $t=\dfrac{7}{2}$ and $t=\dfrac{11}{2}.$ We consider a linear combination of Mellin B-spline kernels defined by (see Fig. 1)
\begin{eqnarray*}
\chi(t)=\dfrac{1}{4}\bar{B}_{2}(2te^{-2})+\dfrac{3}{4}\bar{B}_{2}(2te),
\end{eqnarray*}
where $\bar{B}_{2}$ is given by 
\begin{eqnarray*}
\bar{B}_{2}(t)&=&\left\{
\begin{array}{ll}
1-\log t, &  \ \ \ \ \ 1<t<e \\
\vspace{0.25 cm}
1+\log t, & \ \ \  \ \ \ \dfrac{1}{e}<t<1
\\
\vspace{0.25 cm}
0, & \ \ \  \ \ \ \mbox{otherwise.}
\end{array}
\right.
\end{eqnarray*}

\begin{figure}[h]
\centering
{\includegraphics[width=0.8\textwidth]{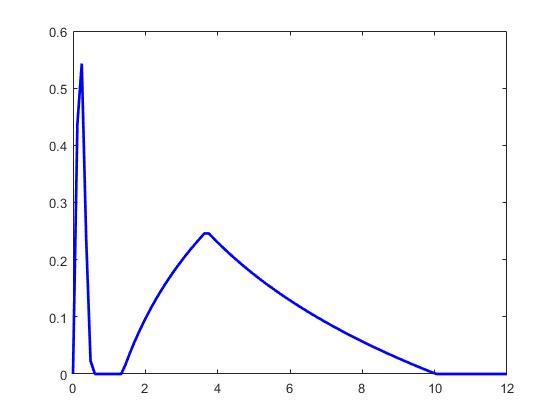}}
\caption{Plot of the kernel $\chi(t)=\dfrac{1}{4}\bar{B}_{2}(2te^{-2})+\dfrac{3}{4}\bar{B}_{2}(2te).$}
\end{figure}
\noindent Clearly the exponential sampling series $S_w^{\chi}{f}$ based on $\chi(t)$ satisfies the conditions (i), (ii) and $\chi(1)=0.$ We also observe that the condition (i) of Theorem \ref{t5} is satisfied with $\alpha=\dfrac{3}{4}.$ From Theorem \ref{t5} and Theorem \ref{t11}, we have that at the discontinuity points of $f,$ the sampling series $S_w^{\chi}{f}$ converges to $\dfrac{3}{4}f(t+0)+\dfrac{1}{4}f(t-0).$ The convergence of the sampling series $S_w^{\chi}{f}$ at discontinuity points $t=\dfrac{3}{2},$ $t=\dfrac{7}{2}$ and $t=\dfrac{11}{2}$ of the function $f$ has been tested and numerical results are presented in Tables 1, 2 and 3. 
\begin{figure}[h]
\centering
{\includegraphics[width=0.8\textwidth]{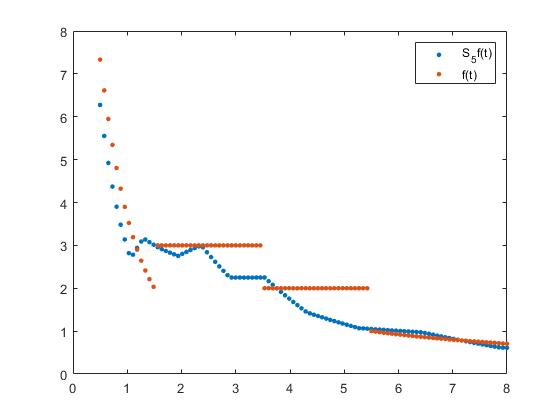}}
\caption{Approximation of $f(t)$ by $S_w^{\chi}{f}$ based on $\chi(t)=\dfrac{1}{4}\bar{B}_{2}(2te^{-2})+\dfrac{3}{4}\bar{B}_{2}(2te)$ for $w=5.$}
\end{figure}

\begin{figure}[h]
\centering
{\includegraphics[width=0.8\textwidth]{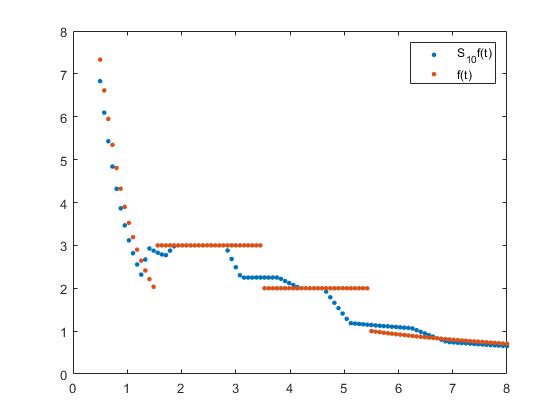}}
\caption{Approximation of $f(t)$ by $S_w^{\chi}{f}$ based on $\chi(t)=\dfrac{1}{4}\bar{B}_{2}(2te^{-2})+\dfrac{3}{4}\bar{B}_{2}(2te)$ for $w=10.$}
\end{figure}

\begin{tb}\label{table1}
 {\it Approximation of $f$ at the jump discontinuity point $t=\dfrac{3}{2}$ by the exponential sampling series $S_w^{\chi}{f}$ based on $\chi(t)$ for different values of $w>0.$ The theoretical limit of 
 $(S_w^{\chi}{f})\left(\dfrac{3}{2}\right)$ as $w\rightarrow \infty$ is
 $\dfrac{3}{4}f\left(\dfrac{3}{2}+0\right)+\dfrac{1}{4}f\left(\dfrac{3}{2}-0\right)=2.75.$}

$  $
\begin{center}
\begin{tabular}{|l|l|l|l|l|l|l|}\hline
 $w$&$ 5$& $10$& $20$& $50$ & $100$& $200$\\
 \hline
 $S_w^{\chi}{f}$&$3.0036$ & $2.8669$ & $2.8059$ & $2.7717$ & $2.7608$ & $2.7554$\\
  \hline

             \end{tabular}
             \end{center}
   \end{tb}

\begin{tb}\label{table2}
 {\it Approximation of $f$ at the jump discontinuity point $t=\dfrac{7}{2}$ by the exponential sampling series $S_w^{\chi}{f}$ based on $\chi(t)$ for different values of $w>0.$ The theoretical limit of 
 $(S_w^{\chi}{f})\left(\dfrac{7}{2}\right)$ as $w\rightarrow \infty$ is 
 $\dfrac{3}{4}f\left(\dfrac{7}{2}+0\right)+\dfrac{1}{4}f\left(\dfrac{7}{2}-0\right)=2.25.$}

$  $
\begin{center}
\begin{tabular}{|l|l|l|l|l|l|l|}\hline
 $w$&$ 5$& $10$& $20$& $50$ & $100$& $200$\\
 \hline
 $S_w^{\chi}{f}$&$2.25$ & $2.25$ & $2.25$ & $2.25$ & $2.25$ & $2.25$\\
  \hline

             \end{tabular}
             \end{center}
   \end{tb}

\begin{tb}\label{table3}
 {\it Approximation of $f$ at the jump discontinuity point $t=\dfrac{11}{2}$ by the exponential sampling series $S_w^{\chi}{f}$ based on $\chi(t)$ for different values of $w>0.$ The theoretical limit of 
 $(S_w^{\chi}{f})\left(\dfrac{11}{2}\right)$ as $w\rightarrow \infty$ is 
 $\dfrac{3}{4}f\left(\dfrac{11}{2}+0\right)+\dfrac{1}{4}f\left(\dfrac{11}{2}-0\right)=1.25.$}

$  $
\begin{center}
\begin{tabular}{|l|l|l|l|l|l|l|}\hline
 $w$&$ 5$& $10$& $20$& $50$ & $100$& $200$\\
 \hline
 $S_w^{\chi}{f}$&$1.0492$ & $1.1420$ & $1.1939$ & $1.2271$ & $1.2384$ & $1.2442$\\
  \hline

             \end{tabular}
             \end{center}
   \end{tb}

\clearpage

{\bf Acknowledgments.} The first two authors are supported by DST-SERB, India Research Grant EEQ/2017/000201. The third author P. Devaraj has been supported by DST-SERB Research Grant MTR/2018/000559.

\end{document}